\documentclass[preprint,12pt]{elsarticle}
\usepackage{amsmath, amssymb}
\newfont {\cyr} {wncyr10}
\usepackage{color}

\usepackage{amsfonts}
\usepackage{algorithm}
\usepackage[noend]{algpseudocode}
\usepackage{amsthm}
\usepackage{mathabx}
\usepackage{xspace}
\usepackage{pdflscape}
\usepackage{afterpage}
\usepackage{capt-of}

\newtheorem{theorem}{Theorem}[section]
\newtheorem{lem}[theorem]{Lemma}
\newtheorem{thm}[theorem]{Theorem}

\newtheorem{rem}[theorem]{Remark}

\newtheorem{defi}[theorem]{Definition}
\newtheorem{ex}[theorem]{Example}

\newcommand{\SX}[1]{\ensuremath{S_{#1}}\xspace}
\newcommand{\Sym}[1]{\ensuremath{Sym(#1)}}

\newcommand{\myfloor}[1]{\left \lfloor #1 \right \rfloor}
\newcommand{\LL}{\mathcal{L}}
\newcommand{\id}{\textrm{id}}
\newcommand{\N}{\mathbb{N}}
\newcommand{\ext}{\operatorname{ext}}
\newcommand{\Min}{\operatorname{Min}}
\newcommand{\Orb}{\operatorname{Orb}}

\newcommand{\Orbcount}{\operatorname{Orbcount}}
\newcommand{\Count}{\operatorname{Count}}
\newcommand{\powset}{\mathcal{P}}

\begin{document}

\begin{frontmatter}

\title{Minimal and canonical images}

\author{Christopher Jefferson}
\ead{caj21@st-andrews.ac.uk}
\ead[url]{http://caj.host.cs.st-andrews.ac.uk/}

\author{Eliza Jonauskyte}
\ead{ej31@st-andrews.ac.uk}

\author{Markus Pfeiffer}
\address{University of St~Andrews\\School of Computer Science\\North Haugh\\St Andrews\\KY16 9SX\\Scotland}
\ead{markus.pfeiffer@st-andrews.ac.uk}
\ead[url]{https://www.morphism.de/~markusp/}

\author{Rebecca Waldecker}
\address{Martin-Luther-Universit\"at Halle-Wittenberg\\Institut f\"ur Mathematik\\06099 Halle\\Germany}
\ead{rebecca.waldecker@mathematik.uni-halle.de}
\ead[url]{http://conway1.mathematik.uni-halle.de/~waldecker/index-english.html}

\journal{Journal of Algebra}

\begin{abstract}
We describe a family of new algorithms for finding the canonical image of a set
of points under the action of a permutation group. This family of algorithms
makes use of the orbit structure of the group, and a chain of subgroups of the
group, to efficiently reduce the amount of search that must be performed to
find a canonical image.

We present a formal proof of correctness of our algorithms and describe experiments
on different permutation groups that compare our algorithms with the previous
state of the art.
\end{abstract}

\begin{keyword}
  Minimal Images, Canonical Images, Computation, Group Theory, Permutation Groups.
\end{keyword}
\end{frontmatter}


\section{Background}

Many combinatorial and group theoretical problems
\cite{DBLP:journals/combinatorics/Soicher99,gent2000symmetry,Distler2012} are
equivalent to finding, given a group \(G\) that acts on a finite set \(\Omega\)
and a subset \(X \subseteq \Omega\), a partition of \(X\) into subsets that are
in the same orbit of \(G\).

We can solve such problems by taking two elements of \(X\) and searching for an element of
\(G\) that maps one to the other. However, this  requires a possible \(O(|X|^2)\) checks,
if all elements of \(X\) are in different orbits.

Given a group \(G\) acting on a set \(\Omega\), a canonical labelling function
 maps each element of \(\Omega\)
to a distinguished element of its orbit under \(G\). Using a canonical labelling function we can check if
two members of \(\Omega\) are in the same orbit by applying the canonical labelling function to both
and checking if the results are equal. More importantly, we can
solve the problem of partitioning \(X\) into orbit-equivalent subsets by performing \(O(|X|)\)
canonical image calculations. Once we have the canonical image of each element, we can organize
the canonical images into equivalence classes by sorting in \(O(|X|log(|X|))\) comparisons, or expected \(O(|X|)\)
time by placing them into a hash table. This is because checking
if two elements are in the same equivalence class is equivalent to checking if their canonical images
are equal.

The canonical image problem has a long history.
Jeffrey Leon \cite{Leon} discusses three types of problems on permutation groups
-- subgroup-type problems (finding the intersection of several
groups), coset-type problems (deciding whether or not the intersection of a
series of cosets is empty, and if not, finding their intersection) and
canonical-representative-type problems. He claims to have an algorithm to efficiently
solve the canonical-representative problem, but does not discuss it further. His comments
have inspired mathematicians and computer scientists
to work on questions related to minimal images and canonical images.

One of the most well-studied canonical-image problems is the canonical graph
problem. Current practical systems derive from partition refinement
techniques, which were first practically used for graph automorphisms by McKay
\cite{McKay80} in the \texttt{Nauty} system. There have been a series of
improvements to this technique, including \texttt{Saucy} \cite{Saucy},
\texttt{Bliss} \cite{JunttilaKaski:ALENEX2007} and \texttt{Traces} \cite{McKay201494}.
A comparison of these systems can be found in \cite{McKay201494}.

We cannot, however, directly apply the existing work for graph isomorphism to
finding canonical images in arbitrary groups. The reason is that McKay's Graph Isomorphism algorithm only considers finding the canonical image of a graph
under the action of the full symmetric group on the set of vertices. Many
applications require finding canonical images under the action of subgroups of
the full symmetric group.

One example of a canonical labelling function is, given a total ordering on \(X\), to
map each value of \(X\) to the smallest element in its orbit under \(G\). This
\emph{Minimal image problem} has been treated by Linton in
\cite{Linton:SmallestImage}. Pech and Reichard \cite{Pech2009} apply techniques
similar to Linton's to enumerate orbit representatives of subsets of $\Omega$
under the action of a permutation group on $\Omega$.
Linton gives a practical algorithm for finding the smallest image of a set under
the action of a given permutation group. Our new algorithm, inspired by Linton's
work, is designed to find canonical images: we extend and generalize Linton's
technique using a new orbit-based counting technique. In this paper we first
introduce some notation and explain the concepts that go into the algorithm,
then we prove the necessary results and finish with experiments that demonstrate
how this new algorithm is superior to the previously published techniques. 

\section{Minimal and Canonical Images}

Throughout this paper, $\Omega$ will be a finite set, $G$ a subgroup of
$\Sym{\Omega}$, and $\Omega$ will be ordered by some (not necessarily total)
order $\leq$.
If $\alpha \in \Omega$, then we denote the orbit of $\alpha$ under $G$ by $\alpha^G$.
Simlarly, if \(A \subseteq \Omega\) and $g \in G$, then 
\(A^g := \{a^g \mid a \in A\}\) and \(A^G := \{ A^g \mid g \in G \}\).

In this paper, we want to efficiently solve the problem of deciding, given two subsets
\(A,B \subseteq \Omega\), if \(A \in B^G\). We do this by defining a canonical image:

\begin{defi}
  A \textbf{canonical labelling function} $C$ for the action of $G$ on a set $\Omega$ is a function
  $C:\powset(\Omega) \rightarrow \powset(\Omega)$ such that, for all $A  \subseteq \Omega$,
  it is true that
  \begin{itemize}
  \item $C(A) \in A^G$, and
  \item ${C(A^g)} = C(A)$ for all $g \in G$.
  \end{itemize}

  In this situation we call $C(A)$ the \textbf{canonical image} of
  $A \subseteq  \Omega$ (with respect to $G$ in this particular action).

  Further, we say that $g_A \in G$ is a \textbf{canonizing element} for \(A\) if and only if
  $A^{g_A} = C(A)$.
\end{defi}

A canonical image can be seen as a well-defined
representative of a $G$-orbit on $\Omega$ with respect to the defined action. While in this
paper we will only consider the action of \(G\) on a set of subsets of \(\Omega\), canonical images are
defined similarly for any group and action.
In practice we want to be able to find canonical images effectively and
efficiently.
In some situations we are interested in computing the canonizing element,
which might not be uniquely determined. Our algorithms will always produce a
canonizing element as a byproduct of search.
We choose to make this explicit here to make the exposition clearer.

Minimal images are a special type of canonical image.

\begin{rem}
  Suppose that $\preccurlyeq$ is a partial order on $\Omega$ such that any
  two elements in the same orbit can be compared by $\preccurlyeq$.

   Let $\Min_\preccurlyeq$ denote the function that, for all $\omega \in \Omega$, maps $\omega$ to the
  smallest element in its orbit. Then $\Min_\preccurlyeq$ is a canonical labelling function.
\end{rem}

In practical applications we are interested in more structure, namely in structures
that $G$ can act on naturally via the action on a given set $\Omega$. These
structures include subsets of $\Omega$, graphs with vertex set $\Omega$, sets of
maps with domain or range $\Omega$, and so on.

In this paper, our main application will be finding canonical images when
acting on a set of subsets of $\Omega$.

\begin{defi}\label{metaorder}
  Suppose that $\leq$ is a total order of $\Omega$.
  Then we introduce a total order $\preccurlyeq$ on $\powset(\Omega)$ as follows:

  We say that $A$ \textbf{is less than} $B$ and write $A \preccurlyeq B$ if and only if $A$
  contains an element $a$ such that $a \notin B$ and $a \leq b$ for all
  $b \in B \setminus A$.
\end{defi}

\begin{ex}
  Let $\Omega:=\{1,2,3,4,5,6,7\}$ with the natural order and let $A:=\{1,3,4\}$,
  $B:=\{3,5,7\}$, $C:=\{3,6,7\}$, $D:=\{1,3\}$ and $E:=\{2\}$.

  Now $A \preccurlyeq B$, because $1 \in A$, $1 \notin B$ and $1$ is smaller
  than all the elements in $B$, in particular those not in $A$. Moreover $A \preccurlyeq C$ for the same reason. Furthermore, $B \preccurlyeq C$,
  because $5 \in B$, $5 \notin C$, and if we look at $C \setminus B$, then
  this only contains the element $6$ and $5$ is smaller.

  Next we consider $A$ and $D$. As $4 \in A \setminus D$ and $D\setminus A=\varnothing$, we see that $A \preccurlyeq D$. Also $A \preccurlyeq E$ because $1 \in A \setminus E$ and $1$ is smaller than all elements in $E \setminus A=E$.
  Finally $E \preccurlyeq B$ because $2 \in E \setminus B=E$ and $2$ is smaller than all elements in $B \setminus E=B$.
\end{ex}

\begin{rem}
The example illustrates that this new order introduced above reduces to lexicographical order for sets of the same size. But for sets of different sizes, it might seem counter-intuitive. Our reason for choosing this different ordering is that it satisfies the
following property:

If \(n \in \N\) and if \(A\) and \(B\) are sets of integers, then \(A \cap \{1,\dots,n\} < B \cap \{1,\dots,n\}\) implies \(A < B\). This means that, when building \(A\) and \(B\) incrementally, we know the order of \(A\) and \(B\) as soon as we find the first integer that is contained in one of the sets but not in the other. This is not true for lexicographic ordering of sets, as \(\{1\} < \{1,2\}\) but \(\{1,1000\} > \{1,2,1000\}\).
\end{rem}

If $G$ is a subgroup of $\Sym{\Omega}$ and $\omega \in \Omega$, then we denote by
$G_\omega$ the point stabilizer of $\omega$ in $G$. For distinct elements
$x,y \in \Omega$, we denote by \(G_{x \mapsto y}\)  the set of all elements of
$G$ that map $x$ to $y$. This set may be empty.

We remark that the above information is readily available from a stabilizer chain
for the group $G$, which can be calculated efficiently. For further details
we refer the reader to \cite{HCG}. We now introduce some notation and then prove a basic result about cosets.

\begin{defi}
 Let $G$ be a permutation group acting on a totally ordered set
 $(\Omega, \leq)$, and let $\preccurlyeq$ denote the induced ordering as explained in Definition \ref{metaorder}.
Let $H$ be a subgroup of $G$ and $S \subseteq \Omega$.
Then we define
  \textbf{the minimal image of $S$ under $H$} to be the smallest element in the set
  $\{S^h \mid h \in H\}$ with respect to $\preccurlyeq$.

In order to simplify notation, we will from now on write $\leq$ for the induced order and then we write $\Min(H,S,\leq)$ for the minimal image of $S$ under $H$.
\end{defi}

\begin{lem}\label{miniprop}
  Let $G$ be a permutation group acting on a totally ordered set
  $(\Omega, \leq)$, and let $H$ be a subgroup of $G$ and $S \subseteq \Omega$.
  Then the following hold for all $x,y \in \Omega$:

\begin{enumerate}

\item
  For all $\sigma \in H_{x \mapsto y}$ it is true that
  $\sigma \cdot H_y=H_{x\mapsto y}=H_x \cdot \sigma$.

\item
  If $\sigma \in H_{x \mapsto y}$, then
  \(\Min(\sigma \cdot H, S, \leq) = \Min(H, S^{\sigma}, \leq)\).
\end{enumerate}
\end{lem}

\begin{proof}
  If $\sigma \in H_{x \mapsto y}$, then multiplication by $\sigma$ from the
  right or left is a bijection on $H$, respectively. For all $\alpha \in H_x$ we
  have that $\alpha \cdot \sigma$ maps $x$ to $y$ and for all $\beta \in H_y$ we
  see that $\sigma \cdot \beta$ also maps $x$ to $y$. This implies the first
  statement.

  For (ii) we just look at the definition: $\Min(\sigma \cdot H, S, \leq)$ denotes
  the smallest element in the
  set $\{S^{\sigma\cdot h} \mid h \in H\}$ and $\Min(H, S^{\sigma}, \leq)$
  denotes the smallest element in the set $\{(S^\sigma)^h \mid h \in H\}$,
  which is the same set.
\end{proof}

\subsection{Worked Example}

We will find minimal, and later canonical, images using similar techniques to
Linton in \cite{Linton:SmallestImage}. This algorithm splits the problem into
small sub-problems, by splitting a group into the cosets of a point stabilizer.
We will begin by demonstrating this general technique with a worked
example.

\begin{ex}\label{ex:minimal}
  In the following example we will look at $\Omega = \{1,2,3,4,5,6\}$,
  the subgroup {$G = \langle (14)(23)(56), (126)\rangle \leq \SX{6}$},
  and $S = \{2,3,5\}$.
  We intend to find the minimal image $\Min(G,S, \leq)$, where the ordering on
  subsets of $\Omega$ is the induced ordering from $\leq$ on $\Omega$ as explained in
  Definition \ref{metaorder}.

  We split our problem into pieces by looking at cosets of \(G_1 = \langle (3,4,5)
  \rangle\). The minimal image of \(S\) under \(G\) will be realized by an
  element contained in (at least) one of the cosets of \(G_1\), so if we find
  the minimal image of \(S\) under elements in each coset, and then take the
  minimum of these, we will find the global minimum.

  Lemma \ref{miniprop} gives that, for all $g \in G$, it holds that $\Min(g \cdot
  G_1, S, \leq) = \Min(G_1, S^g, \leq)$, and
  so we can change our problem from looking for the minimal image of $S$ with
  respect to cosets of \(G_1\) to looking at images of \(S^g\) under elements of
  $G_1$ where $g$ runs over a set of coset representatives of $G_1$ in $G$.

  For each $i \in \{1,\dots,6\}$ we need an element
  $g_i \in G_{i \mapsto 1}$ (where any exist), so that we can then consider $S^{g_i}$.

  We choose the elements $\id$, $(162)$, $(146523)$, $(14)(23)(56)$, $(142365)$
  and $(126)$ and obtain six images of $S$:
  \[
    \{2,3,5\}, \{1,3,5\}, \{3,1,2\}, \{3,2,6\}, \{3,6,1\}, \{6,3,5\}.
  \]

  As we are looking at the images of these sets under \(G_1\), we
  know that all images of a set containing \(1\) will contain \(1\),
  and all images of a set not containing \(1\) will not contain \(1\).
  From Definition \ref{metaorder}, all subsets of \(\{1,\dots,6\}\)
  containing \(1\) are smaller than all subsets not containing \(1\). This means
  that we can filter our list down to $\{1,3,5\}, \{3,1,2\}$ and $\{3,6,1\}$.

  Furthermore, \(G_1\) fixes \(2\), so by the same argument we can filter our
  list of sets not containing \(2\), leaving only \(\{3,1,2\}\). The minimal image
  of this under \(G_1\) is clearly \(\{3,1,2\}\) (in this particular
  case we could of course also have stopped as soon as we saw \(\{3,1,2\}\), as
  this is the smallest possible set of size 3).

  Now, let us consider what would happen if the ordering of the integers was
  reversed, so we are looking for \(\Min(G, S, \geq)\), again with the
  induced ordering.

  For the same reasons as above, we begin by calculating
  $G_6 = \langle (3,5,4) \rangle$ and by finding images of \(S\) for some
  element from each coset of \(G_6\) in $G$.

  An example of six images is
  \[
    \{ 1, 5, 4 \}, \{ 6, 5, 3 \}, \{ 4, 6, 1 \}, \{ 5, 1, 2 \}, \{ 3, 2, 6 \},
    \{ 2, 3, 5 \}.
  \]
  We can ignore anything that does not contain \(6\), so we are left with:

  \[
    \{ 6, 5, 3 \}, \{ 4, 6, 1 \}, \{ 3, 2, 6 \}.
  \]

  As \(5\) is not fixed by \(G_6\), we can not reason about the presence
  or absence of \(5\) in our sets. There is an image of every set that
  contains \(5\), and there are even two distinct images of \(\{6,5,3\}\) that
  contain \(5\). Therefore we must continue our search by considering \(G_{6,5}\).

  Application of an element from each coset of $G_{6,5}$ to $S$ generates nine
  sets, of which four contain the element $5$. In fact we reach \(\{6,3,5\},
  \{6,5,4\}\) from the set $\{ 6, 4, 3 \}$, we reach \(\{5,6,1\}\) from the set
  $\{ 4, 6, 1 \}$ and we reach \(\{5,2,6\}\) from the set $\{ 3, 2, 6 \}$. From
  these we extract the minimal image \(\{6,5,4\}\).

  In this example, different orderings of \(\{1,2,3,4,5,6\}\) produced different sized
  searches, with different numbers of levels of search required.
\end{ex}

\section{Minimal Images under alternative orderings of \(\Omega\)}
\label{sec:alternate ordering}
As was demonstrated in Example \ref{ex:minimal}, the choice of ordering of the set our group acts on
influences the size of the search for a minimal image. In this section we will show
how to create
orderings of $\Omega$ that, on average, reduce the size of search for a minimal image.

We begin by showing how large a difference different orderings can make. We do
this by proving that, for any choice $\preccurlyeq$ of ordering of \(\Omega\), group \(G\) and
any input set $S$, we can construct a minimal image
problem that is as hard as finding \(\Min(G,S,\preccurlyeq)\), but where
reversing the ordering on $\Omega$ makes the problem trivial.

We make this more precise: Given $n \in \N$, a permutation group \(G\) on \(\{1,\dots,n\}\) with some ordering $\leq$ and a subset \(S
\subseteq \{1,\dots,n\}\), we construct a group \(H\) and a set \(T\) such that
\(\Min(G, S, \leq) = \Min(H,T, \leq) \cap \{1..n\}\), which shows that finding
\(\Min(H,T,\leq)\) is at least as hard as finding \(\Min(G,S,\leq)\). On the
other hand, we will show that \(\Min(H,T, \geq) = T\) and that this can be deduced without search.
This is done in Lemma \ref{ex}.
An example along the way will illustrate the construction.

\begin{defi}
We fix $n \in \N$ and we let $k \in \N$. For all $j \in \N$ we define $q(j) \in
\N$ (where $q$ stands for ``quotient'') and $r(j) \in \{1,\dots,n\}$ (where $r$ stands for ``remainder'') such that $j=q(j) \cdot n+r(j)$.

Let $\ext : G \rightarrow \SX{k \cdot n}$ be the following map:
For all $g \in G$ and all $j \in \{1,\dots,k \cdot n\}$, the element
$\ext(g)$ maps $j$ to $q(j) \cdot n+r(j)^g$.
\end{defi}

\begin{ex}
Let $n=4$ and $G=\SX{4}$. Then we extend the action of $G$ to the set
$\{1,\dots,12\}$ using the map $\ext$.

For example $g=(134)$ maps $4$ to $1$.
We write $12=2 \cdot 4+4$ and then it follows that
$\ext(g)$ maps $12$ to $2 \cdot 4 +4^g=8+1=9$.
In fact $g$ acts simultaneously on the three tuples $(1,2,3,4)$, $(5,6,7,8)$
and $(9,10,11,12)$ as it does on $(1,2,3,4)$.
\end{ex}

\begin{defi}
Fixing $n,k \in \N$ and a subgroup $G$ of $\SX{n}$, and using the map $\ext$
defined above, we say that \textbf{$H$ is the extension of $G$ on
  $\{1,\dots,k\cdot n\}$}
if and only if $H=\{\ext(g) \mid g \in G\}$ is the image of $G$
under the map $\ext$.
\end{defi}

The extension $H$ of $G$ on a set $\{1,\dots,k \cdot n\}$
is a subset of $\SX{k \cdot n}$. We show now that even more is true:

\begin{lem}\label{ext}
Let $n,k \in \N$ and $G \le \SX{n}$. Then the extension of $G$ onto $\{1,\dots,k \cdot n\}$ is a subgroup of $\SX{k \cdot n}$ that is isomorphic to $G$.
\end{lem}

\begin{proof} Let $H:=\ext(G)$ be the image of $G$ under the map $\ext$ and let
$a,b \in G$ be distinct. Then let $j \in \{1,\dots,n\}$ be such that $j^a \neq
j^b$.
By definition $\ext(a)$ and $\ext(b)$ map $j$ in the same way that $a$ and $b$
do, so we see that $\ext(a) \neq \ext(b)$. Hence the map $\ext$ is injective.
Therefore $\ext:G \rightarrow H$ is bijective.

Next we let $a,b \in G$ be arbitrary and we let $j \in \{1,\dots,k \cdot n\}$.
Then the composition $ab$ is mapped to $\ext(ab)$, which maps $j$ to $q(j) \cdot
n+r(j)^{ab}$. Now $r(j)^{ab}=(r(j)^a)^b$ and therefore the composition
$\ext(a)\ext(b) \in \SX{k \cdot n}$ maps $j$ to $(q(j) \cdot
n+r(j)^a)^{\ext(b)}=q(j) \cdot n + (r(j)^a)^b$. This is because $r(j)^a \in
\{1,\dots,n\}$.

Hence $\ext(ab)=\ext(a)\ext(b)$. That implies $\ext$ is a group homomorphism
and hence that $G$ and its image are isomorphic.
\end{proof}

\begin{lem}\label{ex}
Let $n \in \N$ and $G \le \SX{n}$. Let $H$ denote the extension of $G$ on
$\{1,\dots,(n+1) \cdot n\}$ and let $S \subseteq \{1,\dots,n\}$. Let
$T:=S \cup \{ l \cdot n + l \mid l \in \{1..n\}\}$, let \(\leq\) denote the natural
ordering of the integers, and let \(\geq\) denote its reverse. For simplicity we use the same symbols for the ordering induced on $\powset(\Omega)$, respectively. Then

\begin{itemize}
\item \(\Min(H,T, \le) \cap \{1,\dots,n\} = \Min(G,S, \le)\).
\item \(\Min(H,T,\ge)=T\).
\end{itemize}
\end{lem}

\begin{proof}
Let $h \in H$. Then by construction $h$ stabilizes the partition
$$[1,\dots,n|n+1,\dots,2n|\dots|n \cdot n,\dots,(n+1)\cdot n].$$ 
Moreover, for all $i \in \{1,\dots,n\}$ and $g \in G$ we have that $i^g=i^{\ext(g)}$
and so Lemma \ref{ext} implies that

\begin{align*}
  \Min(G,S,\le) &= \min_{\le}\{S^g \mid g \in G\}  \cap \{1,\dots,n\}\\
                &= \min_{\le}\{S^{\ext(g)} \mid g \in G\}  \cap \{1,\dots,n\}\\
                &= \min_{\le}\{S^h \mid h \in H\}  \cap \{1,\dots,n\}\\
                &= \Min(H,T, \le) \cap \{1,\dots,n\}.
\end{align*}
This proves the first statement.

For the second statement we notice that $(n + 1)\cdot n$ is now the smallest element
of $T$, and it cannot be mapped to anything smaller, because it also is the
smallest element available.
So if we let $h \in H$ be such that $T^h=\Min(H,T,\ge)$, then $h$ fixes the point $n(n + 1)$. By
definition of the extension, it follows that $h$ also fixes $k \cdot n$ for all
$k \in \{1\ldots n\}$.
The next point of $T$ under the ordering is $n^2 - 1$. It cannot be
mapped by $h$ to $n^2$, because $n^2$ is already fixed, hence $h$ has to fix
$n^2 - 1$, too.

Arguing as above it follows that all points are fixed by $h$, thus in
particular \(\Min(H,T,\ge)=T\), as stated. Furthermore, any algorithm
that stepped through the elements of \(T\) in the order we describe
would find this smallest element without having to perform a branching search,
as at each step there is no choice on which element of \(T\) is the next smallest.
\end{proof}

\subsection{Comparing Minimal Images Cheaply}

We describe some important aspects of Linton's
algorithm for computing the minimal image of a subset of $\Omega$.
\begin{defi}

Suppose that $(\Omega, \le)$ is a totally ordered set and that $G \le
\Sym{\Omega}$. Then \(\Orb(G)\) denotes the list of orbits of $G$ on $\Omega$.
This list of orbits is ordered with respect to the smallest element in each orbit under $\le$.

A $G$-orbit will be called a \textbf{singleton} if and only if it has size $1$.

If \(S \subseteq \Omega\), then we say that a $G$-orbit is \textbf{empty in S}
if and only if it is disjoint from $S$ as a set, and we say that it is
\textbf{full in S} if and only if it is completely contained in $S$.
\end{defi}

\begin{ex}
  Let \(\Omega:=\{1,\dots,8\}\), with the natural ordering on the integers, let
  \(G := \langle (1,4), (2,8), (5,6), (7,8) \rangle\) and let $S:=\{1,3,5,6\}$.

  Then \(\Orb(G) = [ \{1,4\}, \{2,7,8\}, \{3\}, \{5,6\} ]\) because this list
  contains all the $G$-orbits and they are ordered by the smallest element in each
  orbit, namely $1$ in the first, $2$ in the second, $3$ in the third, which is a
  singleton, and $5$ in the last (because $4$ is already in an earlier orbit).

  The orbits $\{3\}$ and $\{5,6\}$ are full in $S$, the orbit $\{2,7,8\}$ is
  empty in $S$ and $\{1,4\}$ is neither.
\end{ex}

\begin{lem}\label{same}
Suppose that $(\Omega, \le)$ is a totally ordered finite set and that $G \le
\Sym{\Omega}$. If
$\Min(G,S,\le)=\Min(G,T,\le)$ and $\omega \in \Omega$, then $\omega^G$ is empty in $S$ if
and only if it is empty in $T$, and $\omega^G$ is full in $S$ if and only if it is
full in $T$.
\end{lem}

\begin{proof}
  Let $\omega \in \Omega$ and suppose that $\omega^G$ is empty in $S$. As $\omega^G$
  is closed under the action of $G$, and $S_0 := \Min(G,S,\le)$ is an image
  of $S$ under the action of $G$, we see that $\omega^G$ is empty in $S_0$ and hence
  in $T_0:=\Min(G,T,\le)$. Thus $\omega^G$ is empty in $T$, which is an image of $T$
  under the action of $G$. The same arguments work vice versa.

  Next we suppose that $\omega^G$ is full in $S$. Then it is full in $S_0=T_0$ and
  hence in $T$, and the same way we see the converse.
\end{proof}

We can now prove Theorem \ref{thm:minsplit}, which provides the main technique
used to reduce search. This allows us to prove that the minimal image of some set \(S\)
will be smaller than or equal to the minimal image of a set \(T\), without explicitly
calculating the minimal image of either \(S\) or \(T\).

\begin{thm}\label{thm:minsplit}
  Suppose that \(G\) is a permutation group on a totally ordered finite set
  \((\Omega, \le)\) and that \(S\) and \(T\)
  are two subsets of \(\Omega\) where \(|S| = |T|\).

  Suppose further that \(o\) is the first orbit in the list $\Orb(G)$ that
  is neither full in both \(S\) and \(T\) nor empty in both \(S\) and \(T\). If
  \(o\) is empty in \(T\), but not in $S$, then $\Min(G,S,\le)$ is strictly
  smaller than $\Min(G,T,\le)$.
\end{thm}

\begin{proof}
Suppose that \(o\) is empty in \(T\), but not in \(S\). Then $o$ is empty in
$T_0:=\Min(G,T,\le)$, but not in $S_0:=\Min(G,S,\le)$, and in particular $T_0$
and $S_0$ are distinct, as we have seen in Lemma \ref{same}.

Let $\alpha$ denote the minimum of the orbit $o$ with respect to $\le$ and let $\omega \in \Omega$. If $\omega < \alpha$,
then
$\omega \notin o$, so the the orbit $\omega^G$ appears in the list $\Orb(G)$
\emph{before} $o$ does. Then the choice of $o$ implies that one of the following
two cases holds:

\begin{enumerate}
\item \(\omega^G\) is full in both \(S\) and \(T\). In particular, for all $g \in G$ we have that
$\omega \in S^g \cap T^g$.
\item \(\omega^G\) is empty in both \(S\) and \(T\). In particular, for all $g \in G$ we have that
$\omega^G \cap S^g=\varnothing$ and $\omega^G \cap T^g=\varnothing$.
\end{enumerate}

If $S_0$ contains an element $\omega\in\Omega$ such that $\omega < \alpha$, then Case (i)
above holds and $\omega \in T_0$. So $S_0 \cap \{\omega' \in \Omega \mid \omega' < \alpha\}= T_0
\cap \{\omega' \in \Omega \mid \omega < \alpha\}$.

Since $S_0$ and $T_0$ are distinct, they must differ amongst the elements
at least as large as $\alpha$ and, since they have the same cardinality, the smallest
such element determines which of $S_0$ and $T_0$ is smaller.

We recall that $o=\alpha^G$ is empty in $T$ and non-empty in $S$, so there exists
some $g \in G$ such that $\alpha \in S^g$. Then $S^g=S_0$ and $\alpha \notin T_0$, so
$S_0$ is strictly smaller that $T_0$.
\end{proof}

Here is an example how to use Theorem \ref{thm:minsplit}.

\begin{ex}
Let $\Omega:=\{1,\dots,10\}$ with natural ordering, and let
$G:=\langle (12), (45),(56), (89)\rangle$. We consider the sets $S:=\{3,6,7\}$
and $T:=\{3,7,9\}$ and we want to calculate the smallest of \(\Min(G,S)\) and
\(\Min(G,T)\). Hence, we want to know which one is smaller as cheaply as
possible, to avoid
superfluous calculations.

We first list the orbits of $G$:
$[\{1,2\},\{3\},\{4,5,6\},\{7\},\{8,9\},\{10\}]$.

Going through the orbits as listed, we see that the first one is empty in $S$
and $T$, the second one is full in $S$ and $T$, and the third one gives a
difference for the first time. It is empty in $T$, but not in $S$, so Theorem
\ref{thm:minsplit} yields that the minimal image of $S$ is strictly smaller than
that of $T$.

\end{ex}

\subsection{Static Orderings of \(\Omega\)}\label{sec:static}

In this section we look at
which total ordering of \(\Omega\) should be used to minimize the amount
of time taken to find minimal images of subsets of $\Omega$.

Given a group \(G\) we will choose an ordering on $\Omega$ such that orbits
with few elements appear as early as possible. In particular, singleton orbits should
appear first.

This is justified by the fact that singleton orbits are always either full or empty. Also,
we would expect smaller orbits to be more likely
to be empty or full than larger orbits. This means that small orbits placed early
in the ordering of \(\Omega\) are
more likely to lead to Theorem \ref{thm:minsplit} being applicable, leading to a
reduction in search.

Algorithm~\ref{alg:minorbit} heuristically chooses a new ordering for an ordered
set \(\Omega\), only depending on the group $G$, under the assumption that the
algorithm that computes minimal images will pick a point from a smallest
non-singleton orbit to branch on.
This will not always be true -- in practice Linton's algorithm branches on the
first orbit which contains some point contained in one of the current candidates
for minimal image.

However, we will show that in Section \ref{sec:ex} that
Algorithm~\ref{alg:minorbit} produces substantially smaller,
and therefore faster, searches in practice.

It is not necessary in Line~\ref{line:point2} of Algorithm~\ref{alg:minorbit} to
choose the smallest element of \(\mathit{Points}\), choosing an arbitrary element will,
on average, perform just as well.
By fixing which point is chosen, we ensure that independent implementations will
produce the same ordering and therefore the same canonical image.

\begin{algorithm}
  \caption{FixedMinOrbit}\label{alg:minorbit}
  \begin{algorithmic}[1]
    \Procedure{MinOrbitOrder}{$\Omega, G$}
    \State{\(\mathit{Remain} := \Omega\)}
    \State{\(\mathit{Order} := []\)}
    \State{\(H := G\)}
    \While{\(|\mathit{Remain}| > 0\)}
		\State{\(\mathit{OrbSize} := \Min\left\{|o|\ \mathrel{\big|}  o \in \Orb(H), o \cap \mathit{Remain} \neq \emptyset\right\}\)}\label{line:orb}
		\State{\(\mathit{Points} := \left\{o \mathrel{\big|} o \in \Orb(H), |o| = \mathit{OrbSize},  o \cap \mathit{Remain} \neq \emptyset\right\}\)}\label{line:point}
		\State{\(\mathit{MinPoint} := \Min\left\{x \mid o \in \mathit{Points}, x \in o\right\}\)}\label{line:point2}
		\State{\(\mathit{Remain} := \mathit{Remain} \backslash \{ \mathit{MinPoint} \}\)}
		\State{\(Add(\mathit{Order}, \mathit{MinPoint})\)}
		\State{\(H := G_{\mathit{MinPoint}}\)}
    \EndWhile
    \State\Return{\(\mathit{Order}\)}
    \EndProcedure
  \end{algorithmic}
\end{algorithm}

We will also consider one simple modification of Algorithm \ref{alg:minorbit}, namely
\texttt{FixedMaxOrbit} (which is the same as \texttt{FixedMinOrbit}) with line
  \ref{line:orb} changed to pick orbits of maximum size.

If our intuition about Theorem \ref{thm:minsplit} is correct, then
\texttt{MaxOrbit} should almost
always produce a larger search than \texttt{MinOrbit} or a random ordering
of \(\Omega\).

\subsection{Implementing alternative orderings of \(\Omega\)}

Having calculated an alternative \(Order\) using \textbf{FixedMinOrbit} or \textbf{FixedMaxOrbit},
we could create a version of \textbf{MinimalImage} which accepted an explicit
ordering. However, rather than editing the algorithm, we can instead perform a
pre-processing step, using Lemma~\ref{lem:map}.

\begin{lem}\label{lem:map}
  Consider a group \(G\) that acts on $\Omega = \{1,\dots,n\}$ and a
  permutation \(\sigma \in \Sym{\Omega}\). We define an ordering \(\leq_\sigma\) on
  \(\{1,\dots,n\}\), where for all $x,y \in \Omega$ we have that \(x \leq_\sigma y\) if and only if
  \(x^\sigma \leq y^\sigma\).

  For the induced orderings $\preccurlyeq$ and $\preccurlyeq_\sigma$
  on subsets of $\Omega$ as in Definition \ref{metaorder}
  it holds that
  \[ X \preccurlyeq_\sigma  Y \Leftrightarrow X^\sigma \preccurlyeq Y^\sigma \]
  for all subsets $X$ and $Y$ of $\Omega$, and hence (simplifying notation)
  \[\Min(G,S,\preccurlyeq_\sigma) = \Min(G^\sigma,S^\sigma,\preccurlyeq)^{\sigma^{-1}}.\]
\end{lem}

\begin{proof}

  Following Definition \ref{metaorder}, $X \preccurlyeq_\sigma Y$ if and only if
  there is an $x \in X$ such that $x\not\in Y$ and for all
  $y \in Y \backslash X$ it holds that $x \leq_\sigma y$.
  By definition of $\leq_\sigma$, this is the case whenever $x^\sigma \leq y^\sigma$,
  and since $x^\sigma \in X^\sigma$ and for all $y^\sigma$ in $Y^\sigma \backslash X^\sigma$ it
  holds that $x^\sigma \leq y^\sigma$, it follows that $X^\sigma \preccurlyeq Y^\sigma$.

  \medskip

  Consider the map $\varphi_\sigma : S^G \rightarrow (S^\sigma)^{G^\sigma}$ that maps sets
  $X \in S^G$ to $X^\sigma \in (S^\sigma)^{G^\sigma}$. This map is bijective, and by the above it respects the ordering, so the second claim follows.

\end{proof}

Lemma~\ref{lem:map} gives an efficient method to calculate minimal images under
different orderings without having to alter the underlying algorithm. The most
expensive part of this algorithm is calculating \(G^\sigma\), but this is still very
efficiently implemented in systems such as GAP, and also can be cached so it
only has to be calculated once for a given \(G\) and \(\sigma\).

\section{Dynamic Ordering of \(\Omega\)}

In Section \ref{sec:static}, we looked at methods for choosing an ordering for
\(\Omega\) that allows a minimal image algorithm to search more quickly.
There is a major limitation to this technique -- it does not make use of
the sets whose canonical image we wish to find.

In this section, instead of producing an ordering ahead of time, we will
incrementally define the ordering of \(\Omega\) as the algorithm
progresses.
At each stage we will consider exactly which extension of our partially
constructed ordering will lead to the smallest increase in the number of sets we
must consider.

We are not free to choose our ordering arbitrarily as we must still map two
sets in the same orbit of \(G\) to the same canonical image.
However, we can use different orderings for sets that are in different
orbits of \(G\).

Firstly, we will explain how we build the orderings that our algorithm uses.


\subsection{Orderings}

When building canonical images, we build orderings as the algorithm progresses.
We represent these partially built orderings as ordered partitions.

\begin{defi}
  Let $k \in \N$ and let \(P=[X_1,\dots,X_k]\) be an ordered partition of
  $\powset(\Omega)$. Then, given two
  subsets $S$ and $T$ of $\Omega$, we write \(S <_P T\) if and only if the cell
  that contains \(S\) occurs before the cell that contains \(T\) in \(P\).

  We say that $P$ is \emph{\(G\)-invariant} if and only if for all
  \(i \in \{1,\dots,k\}\) and \(g \in G\) it holds that \(S \in X_i\)
  if and only if \(S^g \in X_i\).

  A \emph{refinement} of an ordered partition $[X_1,\dots,X_k]$ is an ordered
  partition $[Y_{1,1},Y_{1,2},\dots,Y_{k,l}]$ where $l \in \N$ and such that,
  for all $i \in \{1,\dots,k\}$ and $j \in \{1,\dots,l\}$, we have that
  $Y_{i,j} \subseteq X_i$.

  A \emph{completion} of an ordered partition $X = [X_1,\dots,X_k]$ is a
  refinement where every cell is of size one. Given an ordering for
  \(\Omega\), the \emph{standard completion} of an ordered partition \(X\)
  orders the members of each cell of \(X\) using the ordering on sets from
  Definition \ref{metaorder}.
\end{defi}

In our algorithm we need a completion of an ordered partition, but the exact
completion is unimportant -- it is only important that, given an ordered
partition \(X\), we always return the same completion.
For this reason we define the standard completion of an ordered partition.

\begin{ex}\label{ex:orderpart}
Let $G:=\langle (12)\rangle \le \SX{3}$ and $\Omega :=\{1,2,3\}$. Moreover let

$P:=[ \{\}, \{1\}, \{2\}, \{1,3\}, \{2,3\} \mid \{3\}, \{1,2\}, \{1,2,3\} ]$
be an ordered partition of $\powset(\Omega)$.

The orbits of $G$ on $\Omega$ are $\{1,2\}$ and $\{3\}$.
In particular all elements of $G$ stabilize the partition
$P$.

The ordered partition $Q:=[\{1,3\}, \{2\} \mid \{\}, \{1\}, \{2,3\} \mid \{3\}, \{1,2\}, \{1,2,3\} ]$ is a refinement
of $P$ that is not $G$-invariant.

To see this, we let $g:=(12) \in G$. We have that $\{1,3\}$ is in the first cell
of \(Q\), but $\{1,3\}^g = \{2,3\}$ is not in the first cell.
\end{ex}

In Example~\ref{ex:orderpart} we only considered a very small group, because
the size of \(\powset(\Omega)\) is \(2^{|\Omega|}\). In practice we will not explicitly
create ordered partitions of \(\powset(\Omega)\), but instead store a compact
description of them from which we can deduce the cell that any particular set is in.

In this paper, we will consider two methods of building and refining ordered partitions. We first
define the orbit count of a set, which we will use when building refiners.

\begin{defi}\label{def:orbcount}
  Let \(G\) be a group acting on an ordered set \(\Omega\), and
  \(S \subseteq \Omega\).
  Define the \textbf{orbit count} of \(S\) in \(G\), denoted \(\Orbcount(G,S)\)
  as follows: Given the list \(\Orb(G)\) of orbits of $G$ on $\Omega$ sorted
  by the their smallest member, the list
  $\Orbcount(G,S)$ contains the size of the intersection \(|o \cap S|\) in place
  of $o \in \Orb(G)$.
\end{defi}

We will see the practical use of \(\Orbcount\) in Lemma~\ref{lem:orbcount2}.

\begin{lem}\label{lem:orbcount}
  Suppose that $(\Omega, \le)$ is a totally ordered finite set and that
  $G \le \Sym{\Omega}$. Suppose further that \(S,T \subseteq \Omega\) and that there is some $g \in G$ such that $S^g = T$. Then \(\Orbcount(G,S) = \Orbcount(G,T)\).
\end{lem}

\begin{proof}
  Let $o \in \Orb(G)$ and $g \in G$ with $S^g = T$. Then $o^g = o$ and
  $\alpha \in (o \cap S)$ if and only if $\alpha^g \in (o \cap S)^g$,
  if and only if $\alpha^g \in (o^g \cap S^g) = (o \cap T)$.

\end{proof}

\begin{defi}\label{def:refiners}
  Let \(P\) be an ordered partition of \(\powset(\Omega)\).

  \begin{itemize}
  \item If $\alpha \in \Omega$, then the \textbf{point refinement} of \(P\) by
    $\alpha$ is the ordered partition \(Q\) defined in the following way: Each cell \(X_i\) of \(P\)
    is split into two cells, namely the cell
    {\(\{S \mid S \in X_i, \alpha \in S\}\)}, and the cell
    {\(\{S \mid S \in X_i, \alpha \not\in S\}\)}.
    If one of these sets is empty, then \(X_i\) is not split.

  \item If \(G \le \Sym{\Omega}\) and \(C = \Orbcount(G,T)\) for some set
    \(T \subseteq \Omega\), then the \textbf{orbit refinement} of \(P\) by \(C\)
    is the ordered partition \(Q\) defined as follows: Each cell \(X_i\) of \(P\) is split
    into two cells, namely
    \(\{S\ \mid S \in X_i, \Orbcount(G,S) = C\}\), and
    \(\{S\ \mid S \in X_i, \Orbcount(G,S) \neq C\}\).
    If one of these sets is empty, \(X_i\) is not split.
  \end{itemize}
\end{defi}

\subsection{Algorithm}

We will now present our algorithm. First, we give a technical definition which
will be used in proving the correctness of our algorithm.

\begin{defi}
For all $n \in\N$ we define $\LL_n$ to be the set of lists of length $n$ whose
entries are non-empty subsets of $\Omega$. If $X \in \LL_n$, then as a
convention we write $X_1,\dots,X_n$ for the entries of the list $X$.

If $X \in \LL_n$ and $H \le G \le \SX{n}$ is such that $|G:H|=k \in \N$ and
$Q=\{q_1,\dots,q_k\}$ is a set of coset representatives of $H$ in $G$, then we
define $X^Q$ to be the list whose first $k$ entries are
$X_1^{q_1},\dots,X_1^{q_k}$, followed by $X_2^{q_1},\dots,X_2^{q_k}$ until the last
$k$ entries are $X_n^{q_1},\dots,X_n^{q_k}$. We note that $X^Q \in \LL_{n \cdot
k}$.

Let $X,Y \in \LL_n$. We say that $X$ and $Y$ are $G$-equivalent if and only if
there exist a permutation $\sigma$ of $\{1,\dots,n\}$ and group elements
$g_1,\dots,g_n \in G$ such that, for all $i \in \{1,\dots,n\}$, it holds that
$Y_i=X_{i^\sigma}^{g_i}$.
\end{defi}

We now prove a series of three lemmas about coset representatives, which form
the basis for the correctness proof of our algorithm. They are used to perform
the recursive step, moving from a group to a subgroup.

\begin{lem}\label{rep}
Suppose that $G$ is a permutation group on a set \(\Omega\), that \(H\) is a
subgroup of \(G\) of index $k \in \N$ and that $T$ is a set of left coset
representatives of $H$ in $G$.

Then the following are true:

\begin{enumerate}

\item $|T|=k$.

\item If $T=\{t_1,\dots,t_k\}$ and $g \in G$ and if, for all $i \in \{1,\dots,k\}$,
we define $q_i:=gt_i$, then $Q:=\{q_1,\dots,q_k\}$ is also a set of coset
representatives of $H$ in $G$. In particular there is a bijection from $Q$ to
any set of left coset representatives of $H$ in $G$.

\end{enumerate}
\end{lem}

\begin{proof}
By definition the index of $H$ in $G$ is the number of (left or right) cosets of $H$ in $G$.

For the second statement we let $i,j \in \{1,\dots,k\}$ be such that $q_iH=q_jH$,
hence $gt_iH=gt_jH$. Then $t_j^{-1}t_i=t_j^{-1}g^{-1}gt_i \in H$ and hence
$t_iH=t_jH$. Hence $i=j$ because $t_i$ and $t_j$ are from a set of coset
representatives.
\end{proof}

\begin{lem}\label{equi-buildup}
Suppose that $G$ is a permutation group on a set \(\Omega\) and that $S,T
\subseteq \Omega$ are such that the lists $[S]$ and $[T]$ are $G$-equivalent.

Let \(H\) be a subgroup of \(G\) of index $k \in \N$ and let $P=\{p_1,\dots,p_k\}$
and $Q:=\{q_1,\dots,q_k\}$ be sets of left coset representatives of $H$ in $G$.

Then $[S]^P$ and $[T]^Q$ are \(H\)-equivalent.
\end{lem}

\begin{proof}
As \([S]\) and \([T]\) are \(G\)-equivalent, we know that there exists a group
element $g \in G$ such that $S^g=T$.

We fix $g$, for all $i \in \{1,\dots,k\}$ we let $t_i:=gq_i$ and we consider the
set $T:=\{t_i \mid i \in \{1,\dots,k\}\}$. Then $T$ is also a set of left coset
representatives of $H$ in $G$, by Lemma \ref{rep}. As $P$ is also a set of left
coset representatives, we know that $T$ and $P$ have the same size, so there is
a bijection from $P$ to $T$. This can be expressed in the following way:

There is a permutation $\sigma \in \SX{k}$ such that, for all $i \in \{1,..,k\}$,
it is true that $p_{i^\sigma}H=t_iH$. That means  there is a unique $h_i
\in H$ such that $p_{i^\sigma}h_i=t_i$.

Let now $S_i:=S^{p_i}$ and $T_i:=T^{q_i}$ for $i \in \{1,\dots,k\}$, then
\[
  T_i=T^{q_i}=(S^g)^{q_i}=S^{t_i}=S^{p_{i^\sigma}h_i}=(S^{p_{i^\sigma}})^{h_i}=(S_{i^\sigma})^{h_i},
\]
hence $[S]^P$ and $[T]^Q$ are $H$-equivalent.
\end{proof}

\begin{lem}\label{equi}
  Suppose that $G$ is a permutation group on a set \(\Omega\), that $n \in \N$ and that
  $X,Y \in \LL_n$ are $G$-equivalent. Let \(H\) be a subgroup of \(G\) of index $k \in \N$ and let
  $P=\{p_1,\dots,p_k\}$ and $Q:=\{q_1,\dots,q_k\}$ be sets of left coset representatives of $H$ in $G$.

  Then $X^P$ and $Y^Q$ are \(H\)-equivalent.
\end{lem}

\begin{proof}
As \(X\) and \(Y\) are \(G\)-equivalent, we know that there exist a permutation
$\sigma \in \SX{n}$ and $g_1,\dots,g_n \in G$ such that $Y_i=X_{i^\sigma}^{g_i}$ for all $i \in
\{1,\dots,n\}$. We fix this permutation $\sigma$.

If \(i \in \{1,\dots,n\}\), then \([X_{i^\sigma}]\) and \([Y_i]\) satisfy the
hypothesis of Lemma \ref{equi-buildup}, so it follows that \([X_{i^\sigma}]^P\)
and \([Y_i]^Q\) are $H$-equivalent.

So we find a permutation \(\alpha_i \in \SX{k}\) and group elements
\(h_{i1},\dots,h_{ik} \in H\) such that \(
(X_{i^\sigma}^{p_{j^{\alpha_i}}})^{h_{ij}} = Y_i^{q_i}\) for all $j \in
\{1,\dots,k\}$.

Using \(\sigma\) and $\alpha_1,\dots,\alpha_k$ we define a permutation \(\gamma\)
on \(1,\dots,n \cdot k\).

First we express $l \in \{1,\dots,n \cdot k\}$ uniquely as $l=c_l \cdot k+r_l$ where
$c_l \in \{0,\dots,n-1\}$ and $r_l \in \{1,\dots,k\}$ and we define

\[
  l^\gamma:=(c_l+1)^\sigma \cdot k + r_l^{\alpha_{c_l+1}}.
\]

This is well-defined because of the ranges of $c_l$ and $r_l$ and it is a
permutation because of the uniqueness of the expression and because $\sigma$ and
$\alpha_1,\dots,\alpha_n$ are permutations.

Then, for each $l \in \{1,\dots,n \cdot k\}$, expressed as $l=c_l \cdot k+r_l$ as we did
above, we set $h:=h_{c_l+1,r_l}$, $X'_l:=X_{c_l+1}^{p_{r_l}}$ and
$Y'_l:=Y_{c_l+1}^{q_{r_l}}$.

Then $X^P=[X'_1,\dots,X'_{n \cdot k}]$ and $Y^Q=[Y'_1,\dots,Y'_{n \cdot k}]$.

If we set $p:=p_{r_l^{a_{c_l+1}}}$ and $q:=q_{r_l}$, then
we have, for all $l=c_l \cdot k+r_l$, that

\[
  X_{l^{\gamma}}^{\prime h} = (X_{(c_l+1)^\sigma}^{p})^{h} =
  ((X_{c_l+1}^\sigma)^{p})^{h} = Y_{c_l+1}^{q} = Y'_l.
\]

This is $H$-equivalence.
\end{proof}

We can now describe the algorithm we use to compute canonical images,
and prove that it works correctly.

\begin{defi} \label{def:algparts}
Suppose that $\Omega$ is a finite set, that $G$ is a permutation group on $\Omega$,
that $L \in \LL_k$ and that $P$ is an ordered partition on $\powset({\Omega})$.

  \begin{itemize}
  \item An \emph{$\Omega$-selector} is a function $\mathcal{S}$ such that
    \begin{itemize}
    \item $\mathcal{S}(\Omega, G, L, P) = \omega \in \Omega$, where $|\omega^G| > 1$;
    \item $\mathcal{S}(\Omega, G, L, P) = \mathcal{S}(\Omega, G, M, P)$ whenever
      $L$ and $M$ from $\LL_k$ are $G$-equivalent.
    \end{itemize}

  \item An \emph{Ordering refiner} is a function $\mathcal{O}$ such that for
    all $G$-invariant partitions $P$ of $\powset{({\Omega})}$
    \begin{itemize}
    \item $\mathcal{O}(\Omega, G, L, P) = P'$, where $P'$ is a $G$-invariant
      refinement of $P$;
    \item $\mathcal{O}(\Omega, G, L, P) = \mathcal{O}(\Omega, G, M, P)$ whenever
      $L$ and $M$ from $\LL_k$ are $G$-equivalent.
    \end{itemize}
   \end{itemize}
 \end{defi}

An ordering refiner cannot return a total ordering, unless $G$ acts trivially,
because the partial ordering cannot distinguish between values that are contained in the same orbit of $G$.

Our method for finding canonical images is outlined in Algorithm \ref{alg:canimage}. It recursively searches for the minimal image of a collection of lists,
refining the ordering that is used as search progresses.

\begin{algorithm}
  \caption{$\textsc{CanImage}$}\label{alg:canimage}
  \begin{algorithmic}[1]
    \Require{$\mathcal{S}$ is an $\Omega$-selector, $\mathcal{O}$ is an ordering refiner}
    \Procedure{CanImageRecurse}{$\Omega, G, L, P$}
    \If{$|G| = 1$}\label{line:if}
    \State{$P' := \textrm{Standard completion of } P$}
    \State{\Return{Smallest member of \(L\) under \(P'\)}}
    \EndIf\label{line:endif}
    \State{$H := G_{\mathcal{S}(\Omega,G,L,P)}$}
    \State{$Q := $ coset representatives of $H$ in $G$}
    \State{$P' := \mathcal{O}(\Omega,H,L^Q,P)$}
    \State{$L' := [S \mid S \in L^Q, \not\exists T \in L^Q. T <_{P'} S]$}
    \State\Return{\Call{CanImageRecurse}{$\Omega, H, L', P'$}}
    \EndProcedure{}
  \end{algorithmic}


  \begin{algorithmic}[2]
    \Procedure{CanImageBase}{$\Omega, InputG, InputS$}
    \State\Return{\Call{CanImageRecurse}{$\Omega, InputG, [InputS], [\powset(\Omega)]$}}
  \EndProcedure{}
  \end{algorithmic}
\end{algorithm}

\begin{thm}\label{thm:canorb}
Suppose that $\Omega$ is a finite set, that $G$ is a permutation group on $\Omega$, and
 that $X \subseteq \Omega$. Then $\textsc{CanImageBase}(\Omega,G,X) \in X^G$.
\end{thm}
\begin{proof}
  In every step of Algorithm \ref{alg:canimage} the list of considered sets is
  a list of elements of $X^G$.
\end{proof}

\begin{thm}\label{thm:canimg}
  Let $\Omega$ be a finite set and $G$ a permutation group on $\Omega$.
  Let $X,Y\in\LL_k$ be $G$-equivalent, and let $P$ be a $G$-invariant ordered
  partition of $\powset(\Omega)$. Then
  \[
    \textsc{CanImage}(\Omega,G,X,P) =
    \textsc{CanImage}(\Omega,G,Y,P).
  \]
\end{thm}
\begin{proof}
  We proceed by induction on the size of $G$. 

  The base case is $|G| = 1$. As \(X\) and \(Y\) are \(G\)-equivalent, \(X\) and
  \(Y\) contain the same sets, possibly in a different order. For a given \(P\)
  and \(\Omega\), there is only one standard completion of \(P\) that gives a
  complete ordering on \(\powset(\Omega)\), and so \(X\) and \(Y\) have the same
  smallest element under the standard completion of \(P\) and so the claim
  follows.

  Consider now any non-trivial group $G$, and suppose for our induction hypothesis that the claim holds for
  all groups $H$ where $|H| < |G|$.

  Definition \ref{def:algparts} and the fact that $\mathcal{S}$ is an
  $\Omega$-selector imply that
  \[
    H = G_{\mathcal{S}(\Omega,G,X,P)} = G_{\mathcal{S}(\Omega,G,Y,P)}.
  \]
  Moreover, $H$ is a proper subgroup of $G$. We take two sets $Q_1$ and $Q_2$ of
  coset representatives of $H$ in $G$, which are not necessarily equal.

  Since $\mathcal{O}$ is an ordering refiner, it holds that
  \[
    P' = \mathcal{O}(\Omega,H,X,P) = \mathcal{O}(\Omega,H,Y,P)
  \]

  By Lemma \ref{equi}, \(X^{Q_1}\) and \(Y^{Q_2}\) are \(H\)-equivalent, and by
  definition \(P'\) is $G$-invariant. If we identify the cell \(P'_i\) of \(P'\) that
  contains the smallest element of \(X^{Q_1}\), then \(L'_X\) contains those
  elements of \(X^{Q_1}\) that are in \(P'_i\). Each of these elements is
  \(H\)-equivalent to an element of \(Y^{Q_2}\), and therefore \(L'_X\) is
  \(H\)-equivalent to \(L'_Y\). Then the induction hypothesis yields that
  \[
    \textsc{CanImage}(\Omega, H, X',P) = \textsc{CanImage}(\Omega, G, Y',P),
  \]
  so the claim follows by induction.
\end{proof}

We see from Theorem \ref{thm:canorb} and \ref{thm:canimg} that
$$\textsc{CanImage}(\Omega,G,X,P) = \textsc{CanImage}(\Omega,G,Y,P)
\text{~if and only if~} Y \in X^G.$$

We note that Algorithm \ref{alg:canimage} can easily be adapted to return an
element $g$ of $G$ such that $X^g = \textsc{CanImage}(\Omega,G,X,P)$. This happens by
attaching to each set, when it is created, the permutation that maps it to the
original input \(S\). We omit this addition for readability.


\section{Experiments}\label{sec:ex}

In this section, we will compare how well our new algorithms perform in
comparison with the \textbf{MinImage} function of Linton's.

All of our experiments are performed in GAP~\cite{GAP4} using code available in
the \texttt{Images} package \cite{Images}.
Where our algorithm requires it, we use the implementation of partition
backtracking provided in the \texttt{Ferret} package \cite{Fer}.

We consider a selection of different canonical image algorithms and we analyze how
they perform compared to each other, and compared to the traditional minimal image
algorithm of Linton's, which we will refer to as \textbf{MinImage}.

The first three algorithms that we consider come from Section \ref{sec:alternate
ordering}. They produce, given a group \(G\) on a set \(\Omega\), an ordering of
\(\Omega\). This ordering is then used in \textbf{MinImage}.

\begin{enumerate}
\item \textbf{FixedMinOrbit} uses results from Section \ref{sec:alternate ordering}
  to calculate an alternative ordering of \(\Omega\), choosing small orbits first.
\item \textbf{FixedMaxOrbit} works similarly to \textbf{FixedMinOrbit}, choosing large
  orbits first.
\end{enumerate}

We also consider algorithms that dynamically choose which value to branch on as
search progresses. We will use the following lemma for the proof of correctness
for all our orderings.

\begin{lem}\label{lem:orbcount2}
  Let $\Omega$ be a finite set,  $G$ a permutation group on $\Omega$, and $P$
  an ordered partition of $\powset(\Omega)$. 
  
  Let $L \in \LL_n$, and
  \(\Count = [ \Orbcount(G,S) \mid S \in L ]\).

  \begin{itemize}
  \item Any function that accepts \((\Omega, G, L, P)\) and returns
    some \(\omega \in \Omega\) with $|\omega^G| > 1$ that is invariant
    under reordering the elements of \(\Count\), is an \(\Omega\)-selector.

  \item Any function that accepts \((\Omega, G, L, P)\) and returns
    either \(P\) or the point refinement of \(P\) by some \(\omega^G\) for
    \(\omega \in \Omega\), and is invariant under permutation of the elements
    of \(\Count\), is an ordering refiner.

  \item Any function that accepts \((\Omega, G, L, P)\) and returns either
    \(P\) or the orbit refiner of \(P\) by some member of \(\Count\) and is invariant
    under permutation of the elements of \(\Count\), is an ordering refiner.
\end{itemize}
\end{lem}

\begin{proof}
  The only thing we have to show is that, if $L$ and $M$ are $G$-equivalent, then
  any of the functions above yield the same result for inputs $(\Omega,G,L,P)$
  and $(\Omega,G,M,P)$.
  By Lemma~\ref{lem:orbcount}, any \(G\)-equivalent lists $L$ and $M$ will
  produce the same list \(\Count\), up to reordering of elements, and hence the
  claim follows.
\end{proof}

Firstly we define a list of orderings. Each of these orderings chooses an orbit,
or list of orbits, to branch on -- we will then make an \(\Omega\)-selector by
choosing the smallest element in any of the orbits selected, to break ties (we
could choose any point, as long as we picked it consistently). Each of these
algorithms operates on a list \(L \in \LL_k\). In each case we look for an
orbit, ignoring orbits of size one (as fixing a point that was already fixed
leads to the same group).

Firstly we will consider two algorithms that only consider the group, and not \(L\):

\begin{enumerate}
\item \textbf{MinOrbit} Choose a point from a shortest non-trivial orbit that
  has a non-empty intersection with at least one element of \(L\).
\item \textbf{MaxOrbit} Choose a point from a longest non-trivial orbit that
  has a non-empty intersection with at least one element of $L$.
\end{enumerate}

We also consider four algorithms that consider both the group, and \(L\).

In the following, for an orbit \(o\)
\begin{enumerate}
\item \textbf{RareOrbit} minimises \(\sum\limits_{s \in L}\mid s \cap o|\),
\item \textbf{CommonOrbit} maximises \(\sum\limits_{s \in L}\mid s \cap o|\),
\item \textbf{RareRatioOrbit} minimises \({\log(\sum\limits_{s \in L}\mid s \cap o|)}/{|o|}\),
\item \textbf{CommonRatioOrbit} maximises \({\log(\sum\limits_{s \in L}\mid s \cap o|)}/{|o|}\).
\end{enumerate}

The motivation for \textbf{RareOrbit} is that this is the branch which will lead
to the smallest size of the next level of search -- this exactly estimates the
size of the next level if our ordering refiner only fixed a point in orbit
\(o\). We therefore expect, conversely, \textbf{CommonOrbit} to perform badly and to
produce very large searches.

One limitation of \textbf{RareOrbit} is that it will favour smaller orbits -- in
general we want to minimize the size of the whole search. The idea here is that
if we have two levels of search where we split on an orbit of size two, and each
time create 10 times more sets is equivalent to splitting once on an orbit of
size four, and creating 100 times more sets. \textbf{RareRatioOrbit} compensates
for this. We expect that \textbf{CommonRatioOrbit} is the inverse of this, so we also
expect it to perform badly.

For each of these orderings, we use the ordering refiner that
takes each fixed point of \(G\) in their order in \(\Omega\), and performs a point
refinement by each recursively in turn. By repeated application of
Lemma~\ref{lem:orbcount}, this is a \(G\)-invariant ordering refiner.

We also have a set of orderings which make use of orbit counting. To keep the
number of experiments under control, we used the \textbf{RareOrbit}
strategy in each case to choose which point to branch on next, and we also build an orbit
refiner.

Given an unordered list of orbit counts,
\begin{enumerate}
\item \textbf{RareOrbitPlusMin} chooses the lexicographically smallest one.
\item \textbf{RareOrbitPlusRare} chooses the least frequently occurring orbit
  count list (using the lexicographically smallest to break ties).
\item \textbf{RareOrbitPlusCommon} chooses the most frequently occurring orbit
  count list (using the lexicographically smallest to break ties).
\end{enumerate}

\subsection{Experiments}\label{sec:experiments}

In this section we perform a practical comparison of our algorithms, and the
\textbf{MinImage} algorithm of Linton, for three different families of problems:
grid groups, $m$-sets, and primitive groups.

\subsection{Experimental Design}\label{sec:probdef}

We will consider three sets of benchmarks in our testing. In each experiment,
given a permutation group that acts on \(\{1,\dots,n\}\), we will run
an experiment with each of our orderings to find the canonical image of a set of
size \(\myfloor{\frac{n}{2}}, \myfloor{\frac{n}{4}}\) and
\(\myfloor{\frac{n}{8}}\).

We run our algorithms on a randomly chosen conjugate of each primitive group, to
randomize the initial ordering of the integers the group is defined over. The
same conjugate is used of each group in all experiments, and when choosing a
random subset of size \(x\) from a set \(S\), we always choose the same random
subset. We use a timeout of five minutes for each experiment. We force GAP to build
a stabilizer chain for each of our groups before we begin our algorithm, because
this can in some cases take a long time.

For each size of set and each ordering, we measure three things. The total number
of problems solved, the total time taken to solve all problems, counting
timeouts as 5 minutes, and the number of moved points of the largest group
solved. Our experiments were all performed on an Intel(R) Xeon(R) CPU E5-2640 v4
running at 2.40GHz, with twenty cores. Each copy of GAP was allowed a maximum of
6GB of RAM.

\subsubsection{Grid Groups}

In this experiment, we look for canonical images of sets in grid groups.

\begin{defi}\label{def:gridgroup}
  Let $n \in \N$. The direct product $\SX{n} \times \SX{n}$ acts on the set
  $\{1,\ldots, n\} \times \{1,\ldots, n\}$ of pairs in the following way:

  For all $(i,j) \in  \{1,\ldots, n\} \times \{1,\ldots, n\}$
  and all $(\sigma,\tau) \in \SX{n} \times \SX{n}$ we define

  \[
    {(i,j)}^{(\sigma,\tau)} := (i^{\sigma}, j^{\tau}).
  \]

  The subgroup $G \le \Sym{\{1,\ldots, n\} \times \{1,\ldots, n\}}$ defined by this
  action is called the \textbf{$n \times n$ grid group}.
\end{defi}

We note that, while the construction of the grid group is done by starting with
an $n$ by $n$ grid of points and permuting rows or columns independently of each
other, we actually represent this group as a subgroup of $\SX{n\cdot n}$, and we
do not assume prior knowledge of the grid structure of the action.

We ran experiments on the grid groups for grids of size \(5 \times 5\) to \(100 \times
100\). The results of this experiment
are given in Table 1.

The basic algorithm, \textbf{MinImage}, is only able to solve \(22\) problems
within the timeout. \textbf{FixedMinOrbit} solves \(43\) problems, while being
implemented as a simple pre-processing step to \textbf{MinImage}. The dynamic
\textbf{MinOrbit} is able to solve \(55\) problems, and the best orbit-based
strategy, \textbf{SingleMaxOrbit}, solves \(70\) problems. However the advanced
techniques, which filter by orbit lists, perform much better.
Even ordering by the most common orbit list leads to solving over \(185\)
problems, and the best strategy, \textbf{RareOrbitPlusRare}, solves \(235\)
out of the total \(288\) problems.

Furthermore, for these large groups, the algorithms are still performing very
small searches: For example \textbf{FixedMinOrbit}, on its largest solved
problem with size \(\myfloor{\frac{n}{2}}\) sets and a grid of size \(12 \times 12\), generates
\(793,124\) search nodes, while \textbf{RareOrbitPlusMin} produces only \(183,579\)
search nodes on its largest solved problem with \(\myfloor{\frac{n}{2}}\) sets
(\(65 \times 65\)).

\afterpage{%
\clearpage\thispagestyle{empty}%
\begin{landscape}%
\centering%
\begin{tabular}{|l|l|l|r|r|r|r|r|r|r|r|r|}
\hline
&Stab&\multicolumn{3}{|c|}{\(\myfloor{\frac{n}{2}}\)}&\multicolumn{3}{|c|}{\(\myfloor{\frac{n}{4}}\)}&\multicolumn{3}{|c|}{\(\myfloor{\frac{n}{8}}\)} \\
&Search&\# solved&largest&time&\# solved&largest&time&\# solved&largest&time\\
\hline
RareOrbitPlusRare              & F &  56 & 4,225 & 12,549&  88 & 8,464 & 10,474&  91 & 9,216 & 13,663\\
RareOrbitPlusMin               & F &  54 & 4,225 & 13,207&  89 & 9,604 & 10,546&  90 & 9,025 & 13,484\\
RareOrbitPlusCommon            & F &  30 & 2,209 & 19,306&  65 & 5,476 & 15,408&  90 & 9,216 & 14,009\\
SingleMaxOrbit                 & F &  11 &   225 & 24,671&  27 &   961 & 25,795&  52 & 3,136 & 23,532\\
RareOrbit                      & F &   8 &   144 & 24,763&  17 &   625 & 27,213&  34 & 4,096 & 25,979\\
MinOrbit                       & F &   8 &   144 & 24,739&  16 &   400 & 27,735&  31 & 1,521 & 26,706\\
CommonRatioOrbit               & F &   7 &   121 & 24,990&  15 &   400 & 28,255&  29 & 1,296 & 27,322\\
FixedMinOrbit                  & T &   8 &   144 & 24,796&  13 &   361 & 28,619&  22 &   841 & 28,841\\
FixedMinOrbit                  & F &   8 &   144 & 24,786&  13 &   361 & 28,573&  21 &   841 & 28,859\\
MinImage                        & F &   4 &    64 & 25,819&   8 &   144 & 30,039&  10 &   196 & 31,851\\
MinImage                        & T &   4 &    64 & 25,822&   7 &   144 & 29,962&  10 &   196 & 31,829\\
RareRatioOrbit                 & F &   3 &    49 & 26,126&   5 &    81 & 30,379&   8 &   144 & 32,804\\
MaxOrbit                       & F &   3 &    49 & 26,123&   5 &    81 & 30,373&   8 &   144 & 32,771\\
FixedMaxOrbit                  & F &   3 &    49 & 26,120&   5 &    81 & 30,374&   8 &   144 & 32,725\\
FixedMaxOrbit                  & T &   3 &    49 & 26,113&   5 &    81 & 30,381&   8 &   144 & 32,450\\
CommonOrbit                    & F &   3 &    49 & 26,131&   5 &    81 & 30,396&   8 &   144 & 32,826\\
\hline
\end{tabular}
\label{gridtable}
\captionof{table}{Finding canonical images in grid groups}
\end{landscape}
\clearpage%
}

\subsubsection{M-Sets}

Linton \cite{Linton:SmallestImage} considers, given integers \(n\) and \(m\),
defining a permutation group on the set \(T\) of all subsets of size \(m\) of \(\{1,\dots,n\}\)
under the action on \(T\) of \(\SX{n}\) acting on the
members of the \(m\)-sets. He then looks for minimal images of randomly chosen
subsets of \(T\) of size \(k\), under the standard lexicographic ordering on sets.

We ran experiments for \(m = 2\) and \(n \in \{10,15,\dots,100\}\),
for \(m = 4\) and \(n \in
\{10,15,\dots,35\}\), for \(m=6\) and \(n \in \{10,15,20\}\) and
finally \(m=8\), \(n=10\) as described in Section \ref{sec:probdef}. We choose these 30
experiments as these were the problems that any of our techniques were able
to solve in under 5 minutes. The results of our experiments are shown in Table 2.

Similarly to our experiments on grid groups, we find that the standard
\textbf{MinImage} algorithm is only able to solve a very small set of benchmarks.
Some of the better algorithms, including \textbf{FixedMinOrbit},
are able to solve \(9\) problems. In particular, once again \textbf{MinOrbit}
is not significantly better than \textbf{FixedMinOrbit}, although it is slightly
faster on average over all problems.

However, the orbit-based strategies do much better, solving all the problems
which we set. In the case of sets that contain an eighth of all \(m\)-sets, the best
technique is able to solve all problems any technique can solve, in under 5
minutes. The largest solved problem, which was instance \(n=35, m=4\) for a set
on an eighth of all \(m\)-sets, is solved in only \(6,594\) search nodes by
\textbf{RareOrbitPlusMin}, while the largest solved problem of \textbf{MinImage},
\(n = 15, m = 4\) takes \(631,144\) search nodes.

\afterpage{%
\clearpage\thispagestyle{empty}%
\begin{landscape}%
\centering%
\begin{tabular}{|l|l|l|r|r|r|r|r|r|r|r|r|}
\hline
&Stab&\multicolumn{3}{|c|}{\(\myfloor{\frac{n}{2}}\)}&\multicolumn{3}{|c|}{\(\myfloor{\frac{n}{4}}\)}&\multicolumn{3}{|c|}{\(\myfloor{\frac{n}{8}}\)} \\
&Search&\# solved&largest&time&\# solved&largest&time&\# solved&largest&time\\
\hline
RareOrbitPlusMin               & F &  28 & 27,405 & 1,290&  30 & 52,360 &   835&  30 & 52,360 &   251\\
RareOrbitPlusRare              & F &  27 & 12,650 & 1,235&  30 & 52,360 & 1,107&  30 & 52,360 &   322\\
RareOrbitPlusCommon            & F &  27 & 12,650 & 1,819&  30 & 52,360 & 1,189&  30 & 52,360 &   335\\
MinOrbit                       & F &   9 & 1,365 & 6,701&  22 & 6,435 & 4,143&  28 & 27,405 & 1,203\\
RareOrbit                      & F &   9 & 1,365 & 6,769&  21 & 6,435 & 4,190&  28 & 27,405 & 1,276\\
FixedMinOrbit                  & F &   9 & 1,365 & 6,671&  22 & 6,435 & 4,292&  27 & 12,650 & 1,765\\
FixedMinOrbit                  & T &   9 & 1,365 & 6,678&  22 & 6,435 & 4,291&  27 & 12,650 & 1,766\\
CommonRatioOrbit               & F &   9 & 1,365 & 6,782&  21 & 6,435 & 4,215&  28 & 27,405 & 1,348\\
MinImage                        & F &   4 &   210 & 7,866&   5 &   210 & 7,562&   7 & 1,365 & 7,033\\
MinImage                        & T &   4 &   210 & 7,843&   5 &   210 & 7,573&   7 & 1,365 & 7,037\\
SingleMaxOrbit                 & F &   4 &   210 & 7,813&   5 &   210 & 7,554&   6 & 1,365 & 7,325\\
FixedMaxOrbit                  & T &   4 &   210 & 7,853&   5 &   210 & 7,691&   6 &   210 & 7,569\\
RareRatioOrbit                 & F &   4 &   210 & 7,923&   5 &   210 & 7,642&   5 &   210 & 7,500\\
MaxOrbit                       & F &   4 &   210 & 7,915&   5 &   210 & 7,695&   5 &   210 & 7,500\\
FixedMaxOrbit                  & F &   4 &   210 & 7,890&   5 &   210 & 7,666&   5 &   210 & 7,500\\
CommonOrbit                    & F &   4 &   210 & 7,947&   5 &   210 & 7,724&   5 &   210 & 7,500\\
\hline
\end{tabular}
\label{table:mset}
\captionof{table}{Finding canonical images in M-Set groups}
\end{landscape}
\clearpage%
}

\subsubsection{Comparison to Graph Canonical Image}

A set of \(2\)-sets can be viewed as an undirected graph, where the two sets represent the edges. The problem of finding the canonical image of this set of \(2\)-sets is equivalent to the traditional problem of finding a canonical image of this graph. We can therefore perform a comparison between our technique and Nauty, for these problems. Nauty is able to a find canonical image for all our \(2\)-set problems almost instantly. We investigated why Nauty was able to outperform us by such a large margin, and found three problems. We list the most important one first.

\begin{itemize}
\item The central algorithm of Nauty makes use of properties of the form ``vertices with \(i\) neighbors can only map to other vertices with \(i\) neighbors''. Our algorithm does not make use of this property, as it represents a much more complex condition when considered on \(m\)-sets. Further, while we could add a special case specifically for when the group we are considering is the symmetric group operating on \(m\)-sets, we would prefer to find a more general technique.
\item Our algorithm spends a large proportion of its time calculating stabilizer chains, and mapping sets through elements of the group. This is not required for the graphs.
\item Our algorithm is written in GAP rather than highly optimized C.
\end{itemize}

The most important results to draw from this comparison is that our algorithms should not be viewed as a replacement for graph isomorphism algorithms. We are investigating how to close this performance gap, without special casing.

\subsubsection{Primitive Groups}

In this experiment we look for canonical images of sets under the action of primitive groups
which move between \(2\) and \(1,000\) points. We remove the natural alternating and
symmetric groups, as finding minimal and canonical images in these groups is
trivial and can be easily special-cased.
So we look at a total number of \(5,845\) groups, each of which was successfully treated by at least one algorithm.

We perform the experiment as described in Section \ref{sec:probdef}. The results
are given in Table 3.

All algorithms are able to solve a large number of problems. This is unsurprising
as many primitive groups are quite small (for example the cyclic groups), so any
technique is able to brute-force search many problems. However, we can still see that,
for the hardest problems \(\myfloor{\frac{n}{2}}\), many algorithms
outperform \textbf{MinImage}, and the techniques that use extra orbit-counting
filtering solve 300 more problems, and they run much faster.

For the easiest set of problems, \(\myfloor{\frac{n}{8}}\), we see that the
algorithm \textbf{RareOrbitPlusMin}, which usually performs best, solves slightly fewer problems. This
is because there are a small number of groups where the extra filtering provides
no search reduction, but still requires a small overhead in time. However, the
total time taken is still much smaller, and the algorithm only fails to solve
five problems. These five problems involve groups that are isomorphic to the
affine general linear groups $\operatorname{AGL}(8,2)$,
$\operatorname{AGL}(6,3)$ and $\operatorname{AGL}(9,2)$, and
the projective linear group $\operatorname{PSL}(9,2)$. This suggests that the linear
groups may be a source of hard problems for canonical image algorithms in the future.

\afterpage{%
\clearpage\thispagestyle{empty}%
\begin{landscape}%
\centering%
\begin{tabular}{|l|l|l|r|r|r|r|r|r|r|r|r|}
\hline
&Stab&\multicolumn{2}{|c|}{\(\myfloor{\frac{n}{2}}\)}&\multicolumn{2}{|c|}{\(\myfloor{\frac{n}{4}}\)}&\multicolumn{2}{|c|}{\(\myfloor{\frac{n}{8}}\)} \\
&Search&\# solved&time&\# solved&time&\# solved&time\\
\hline
RareOrbitPlusMin               & F & 5,689 & 64,983& 5,749 & 39,564& 5,840 & 10,287\\
RareOrbitPlusRare              & F & 5,656 & 81,814& 5,738 & 43,202& 5,825 & 19,025\\
RareOrbitPlusCommon            & F & 5,561 & 113,011& 5,721 & 59,740& 5,816 & 25,392\\
FixedMinOrbit                  & T & 5,360 & 220,076& 5,623 & 82,684& 5,817 & 18,295\\
FixedMinOrbit                  & F & 5,354 & 251,786& 5,628 & 99,253& 5,816 & 30,193\\
MinOrbit                       & F & 5,348 & 263,053& 5,641 & 104,241& 5,844 & 27,720\\
RareOrbit                      & F & 5,324 & 272,631& 5,632 & 105,537& 5,844 & 29,365\\
CommonRatioOrbit               & F & 5,323 & 277,050& 5,629 & 107,561& 5,844 & 32,123\\
SingleMaxOrbit                 & F & 4,811 & 465,908& 5,250 & 253,467& 5,648 & 112,147\\
MinImage                        & T & 4,723 & 390,334& 5,163 & 242,477& 5,631 & 88,048\\
MinImage                        & F & 4,710 & 501,952& 5,180 & 280,686& 5,633 & 106,983\\
FixedMaxOrbit                  & F & 4,659 & 514,618& 5,119 & 298,321& 5,587 & 132,508\\
FixedMaxOrbit                  & T & 4,674 & 392,753& 5,095 & 251,001& 5,583 & 107,433\\
MaxOrbit                       & F & 4,641 & 544,222& 5,104 & 310,402& 5,587 & 144,182\\
RareRatioOrbit                 & F & 4,614 & 559,690& 5,090 & 316,609& 5,586 & 152,201\\
CommonOrbit                    & F & 4,604 & 569,047& 5,086 & 319,288& 5,586 & 154,142\\
\hline
\end{tabular}
\label{table:prim}
\captionof{table}{Finding canonical images in Primitive groups}
\end{landscape}
\clearpage%
}

\subsubsection{Experimental Conclusions}

Our experiments show that using \textbf{FixedMinOrbit} is almost always superior
to \textbf{MinImage}. As implementing \textbf{FixedMinOrbit} requires a fairly small
amount of code and time over \textbf{MinImage}, this suggests that any implementations of
Linton's algorithm should have \textbf{FixedMinOrbit} added, because this provides a substantial
performance boost, for relatively little extra coding.

Algorithms that dynamically
order the underlying set, such as \textbf{MinOrbit} and \textbf{RareOrbit}
provide only a small benefit over \textbf{FixedMinOrbit}. Algorithms which add
orbit counting provide a much bigger gain, often allowing solving problems on
groups many orders of magnitude larger than before, thereby greatly advancing the state
of the art.

\section{Conclusions}

We present a general framework and a new set of algorithms for finding the
canonical image of a set under the action of a permutation group. Our
experiments show that our new algorithms outperform the previous state of the art,
often by orders of magnitude.

Our basic framework runs on the concept of refiners and selectors and is not
limited to finding only canonical images of subsets of \(\Omega\). In future
work we will investigate families of refiners and selectors that allow finding
canonical images for many other combinatorial objects.

\vspace{1cm}

\textbf{Acknowledgements.}

All authors thank the DFG (\textbf{Wa 3089/6-1}) and the EPSRC CCP CoDiMa
(\textbf{EP/M022641/1}) for supporting this work. The first author would like to
thank the Royal Society, and the EPSRC (\textbf{EP/M003728/1}). The third
author would like to acknowledge support from the OpenDreamKit Horizon 2020
European Research Infrastructures Project (\#676541). The first and third author
thank the Algebra group at the Martin-Luther Universit\"at
Halle-Wittenberg for the hospitality and the inspiring environment. The fourth
author wishes to thank the Computer Science Department of the University of
St~Andrews for its hospitality during numerous visits.

\bibliography{canonical}
\bibliographystyle{plain}
\end{document}